\documentclass{siamonline220329}

\usepackage{hyperref}
\usepackage[numbers,sort]{natbib}
\usepackage{amsfonts,amssymb}
\usepackage{cleveref}
\usepackage{xcolor}
\usepackage{graphicx}

\newcommand{\real}{\mathbb{R}}

\newcommand{\Ma}[1]{\left\langle{#1}\right\rangle}
\newcommand{\M}[1]{\left({#1}\right)}
\newcommand{\Mb}[1]{\left[{#1}\right]}

\DeclareMathOperator{\supp}{supp}
\renewcommand{\Re}{\text{Re}\,}

\newtheorem{remark}{Remark}

\title{Imaging with thermal noise induced currents\thanks{Submitted to the editors May 8 2023.
\funding{This work was partially funded by the National Science Foundation grants DMS-2008610 and DMS-2136198.}}}
\author{Trent DeGiovanni\thanks{Mathematics Department, University of Utah, Salt Lake City, UT 84112 (\email{degiovan@math.utah.edu}, \email{fguevara@math.utah.edu}).} \and Fernando Guevara Vasquez\footnotemark[2] \and China Mauck\thanks{Formerly: Mathematics Department, University of Utah, Salt Lake City, UT 84112. Currently: STV Incorporated, 200 W Monroe St {\#}1650, Chicago, IL 60606.}}
\begin{document}
\maketitle
 \begin{abstract}
    We use thermal noise induced currents to image the real and imaginary parts of the conductivity of a body. Covariances of the thermal noise currents measured at a few electrodes are shown to be related to a deterministic problem. We use the covariances obtained while selectively heating the body to recover the real power density in the body under known boundary conditions and at a known frequency. The resulting inverse problem is related to acousto-electric tomography, but where the conductivity is complex and only the real power is measured. We study the local solvability of this problem by determining where its linearization is elliptic. Numerical experiments illustrating this inverse problem are included.
 \end{abstract}

 \headers{Imaging with thermal noise induced currents}{T. DeGiovanni, F. Guevara~Vasquez, and C. Mauck}


 \begin{MSCcodes} % MSC2020, see e.g. https://zbmath.org/classification/
35R30, % Inverse problems for PDEs
35J25, % Boundary value problems for second-order elliptic equations
35Q61  % Maxwell equations
 \end{MSCcodes}

%%%%%%%%%%%%%%%%%%%%%%%%%%%%%%%%%%%%%%%%%%%%%%%%%%%%%%%%%%%%%%%%%%%%%%%
\section{Introduction}\label{sec:intro}

In an electrical conductor, the excitement of charge carriers due to heat produces random currents. This phenomenon is called Johnson-Nyquist noise and was first observed in the early 20th century \cite{Johnson:1928:TAE,Nyquist:1928:TAE}. Given a single component with impedance $Z(\omega)$ (in Ohms) at an angular frequency $\omega$ (in $2\pi$ Hz), the variance $\Ma{|J(\omega)|^2}$  of the random currents (in $\textrm{A}^2$) is 
\begin{equation}\label{eqn:Jnnoise}
\Ma{|J(\omega)|^2}=\frac{2 \kappa T}{\pi}\frac{\Re(Z(\omega))}{|Z(\omega)|^2}  \Delta \omega.
\end{equation}
Here $\kappa \approx 1.36 \times 10^{-23} \textrm{J} \cdot \textrm{K}^{-1}$ is Boltzmann's constant, $T$ is temperature (in Kelvin) and $\Delta \omega$ is a bandwidth of interest around $\omega$. Although this noise in a nuisance in electrical circuits, we show one way to use it to image the conductive properties of a body. Johnson-Nyquist noise and its generalization to the Maxwell equations (see, e.g., \cite{Rytov:1989:PSR3}) are examples of a more general physical principle called the fluctuation dissipation theorem, which relates the variance of fluctuations of a linear system about an equilibrium to the dissipative properties of the system, see e.g. \cite{Kubo:1966:FDT,Rytov:1988:PSR2,Zwanzig:2001:NSM}. 

To image the conductivity of a body, we propose heating the body while simultaneously measuring the variance of the thermal noise induced currents using electrodes that are connected to the ground. For instance, this can be done on a two-dimensional conductive body as illustrated in \cref{fig:exp_setup}, with electrodes on its boundary that are connected to the ground (zero voltage or potential).  The electrical measurements are made while the body is heated at a known spatial location via an external source, e.g. a laser, and the process is repeated at different locations to scan the body. In our approach, we also need to subtract measurements of thermal noise induced currents at a known and constant background temperature.

\begin{figure}[h]
\centering
\includegraphics[width=.2\textwidth]{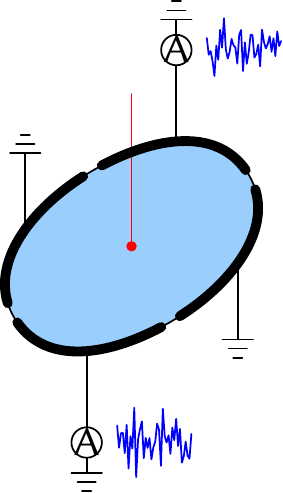}
\caption{A two-dimensional conductive body is attached to the ground via electrodes which are used to measure the thermal noise currents resulting from  heating the body at particular locations, e.g. the location depicted in red.}
\label{fig:exp_setup}
\end{figure}

Our main contribution is to show that such thermal noise current measurements are equivalent to measuring the real power dissipated inside the conductive body, i.e.
\begin{equation}\label{eqn:func}
\sigma'(x)|\nabla u(x)|^2,
\end{equation}
where $\sigma'(x)$ is the real part of the conductivity and $u$ solves an appropriate (deterministic) auxiliary problem which depends on electrodes configuration and the conductivity $\sigma(x)$ \textit{which can be complex} (see \cref{sec:det}). 

%%%%%%%%%%%%%%%%%%%%%%%%%%%%%%%%%
\subsection{Related work}
Recovering $\sigma'(x)$ from functionals of the form \eqref{eqn:func} is well-studied for the case of real $\sigma$ in ultrasound modulated electrical impedance tomography or acousto-electric tomography \cite{Ammari:2008:EIT}, where the internal functional \eqref{eqn:func} is measured by locally perturbing a conductive body using ultrasound waves, while making electrical measurements on the body's surface. Various reconstruction approaches have since been studied for this problem \cite{Capdeboscq:2009:IMN,Kuchment:2010:SFU, Kuchment:2011:RAE,Bal:2013:LMI,Bal:2013:IDK} as well as its well-posedness \cite{Gebauer:2008:IAT,Kuchment:2012:SIP,Bal:2013:CPU}. The problem of recovering an anisotropic real conductivity has also been studied \cite{Bal:2014:LIF, Monard:2012:IAD,Monard:2012:IDP,Monard:2013:IAC}. A similar problem using microwaves instead of ultrasound is discussed in \cite{Ammari:2011:MIE}. Optical tomography can also be modulated by ultrasound, allowing measurements of a functional similar to \eqref{eqn:func}, see \cite{Bal:2010:ISA,Powell:2016:GBQ}. Hybrid inverse problems (including acousto-electric tomography) have been studied by formulating them as an overdetermined system of non-linear partial differential equations and then studying their local uniqueness properties by linearizing, see \cite{Kuchment:2012:SIP,Bal:2014:HIP}. For reviews on hybrid inverse problems, see \cite{Ammari:2008:IME,Bal:2013:HIP}. 

The $\sigma$ complex case is considered in \cite{Bal:2013:RCS}, but this analysis only applies if the fields are known (as in elastography). Complex $\sigma$ were also considered in the case of the Maxwell equations in \cite{Zhou:2014:HIP}. However, to our knowledge, there is no study of the functional \eqref{eqn:func} where $u$ depends on a complex $\sigma$, but only its real part $\sigma'$ appears explicitly in the measurements. 

%%%%%%%%%%%%%%%%%%%%%%%%%%%%%%%%%%%%%%%
\subsection{Possible applications}
The biggest challenge to the applicability of the method that we present here is that the thermally induced random currents are very small and this could introduce signal-to-noise issues. We envision two possible applications. 

The first possible application would be to Atomic Force Microscopy (AFM). In this imaging modality, a height-map of a sample is obtained by measuring the deflections of a cantilever as its tip scans the sample. A heated cantilever tip can be used to heat the sample locally without touching it, see e.g. \cite{King:2013:HAF}. Moreover, electrical measurements of thermal noise induced currents can be done simultaneously with the AFM scan. An advantage of this approach is that one can measure height and conductivity of the sample without touching the sample, possibly making the cantilever tip last longer. We mention that conductivity variations can be measured in AFM by creating a voltage difference between the sample and the cantilever (assuming it is conductive).  This method is known as Conductive Atomic Force Microscopy or CAFM, see e.g. the review in \cite{Lanza:2017:CAF}.

The second possible application is to monitoring of laser welding, see e.g. \cite{Cline:1977:HTM}. If two sheets of metal are been welded together, their temperature is raised significantly near the weld and the sheets also become electrically connected. We believe that by measuring thermal noise currents, one can monitor whether the weld was effective. We give an order of magnitude of the signals and background that would need to be measured in this situation in \cref{rem:measure_size}.

%%%%%%%%%%%%%%%%%%%%%%%%%%%%%%%%%%%%%%%
\subsection{Contents}
We start in \cref{sec:quasi} by deriving a quasi-static model from the Maxwell equations with a current source modeling the random currents. In \cref{sec:det} we show how variances of the random currents are related to a deterministic problem. This is done for two different kinds of boundary conditions. Moreover we give rough magnitude estimates for the currents that would need to be measured to implement our approach.  In \cref{sec:stability}, we analytically and numerically analyze the linearized real problem (\cref{sec:real_prob}) and linearized complex problem (\cref{sec:comp_prob}). This analysis is based on \cite{Bal:2014:HIP} where the ellipticity, in the  Douglis-Nirenberg sense \cite{Douglis:1955:IEE}, is established for the real linearized problem. We give a condition in \cref{lem:sym} for the linearized problem with complex conductivity to be elliptic in the Douglis-Nirenberg sense \cite{Douglis:1955:IEE}. Still, it remains unclear if boundary conditions exist such that the fields associated with the auxiliary problem satisfy this condition. Then in \cref{sec:numrec} we present a simple numerical reconstruction approach based on a finite difference discretization of the problem (\cref{sec:disc_model}). We solve the inverse problem using data that either comes directly from the internal functional \eqref{eqn:func} or from simulated realizations of random currents. In addition, we show reconstructions in the case that conductivity is  real (\cref{sec:real}) or complex (\cref{sec:complex}). Finally, we summarize our results in \cref{sec:summary}. 

%%%%%%%%%%%%%%%%%%%%%%%%%%%%%%%%%%%%%%%%%%%%%%%%%%%%%%%%%%%%%%%%%%%%%%%
\section{The quasi-static model}
\label{sec:quasi}
In an isotropic medium, thermal fluctuations induce fluctuation of
charge carriers near an equilibrium. For non-magnetic media, the thermal fluctuation
currents can be modeled by a random external electric current $j_e$ ($\textrm{A}/\textrm{m}^2)$ in
the Maxwell equations \cite{Rytov:1989:PSR3}, namely
\begin{equation}
\begin{aligned}
\nabla \times H &= i \omega \varepsilon E + j_e\\
\nabla \times E &= -i \omega \mu H.
\end{aligned}
\label{eq:maxwell}
\end{equation}
Here $E$ and $H$ are the electric and magnetic fields, and the angular frequency is $\omega$. The convention for time harmonic fields here is that $\mathcal{E}(x,t) = \Re\left[E(x,\omega) \exp[ i \omega t]\right]$. The electric
permittivity is $\varepsilon$ and may be written as $\varepsilon =
\varepsilon' - i
\sigma' / \omega$, where $\varepsilon' \equiv \Re \varepsilon$ and $\sigma'$
is the real conductivity. The magnetic permeability $\mu$ is assumed
real and equal to that of the vacuum. Note that if $\mu$ had an imaginary part (i.e. non-zero magnetic losses), then an analogous ``random magnetic current'' needs to be added to the Maxwell
equations. The fluctuation dissipation theorem (see e.g. \cite[Chapter 1]{Zwanzig:2001:NSM}
and the particular application to the Maxwell equations in
\cite{Rytov:1989:PSR3}) states that the random current field $j_e$
 has zero mean $\Ma{j_e} = 0$ and its
covariance (at a fixed frequency) is 
\begin{equation}
\Ma{j_e(x) j_e^*(x')} = \frac{\kappa}{\pi} T(x) \sigma'(x)
\delta(x-x') \mathbb{I} \Delta \omega,
\label{eq:currcorr}
\end{equation}
where $T(x)$ is the temperature in Kelvin at a point $x$,  $\mathbb{I}$ is the identity matrix and $\Delta \omega$ is a frequency band of interest around $\omega$. We emphasize that \eqref{eq:currcorr} depends only on the temperature
and the real part of the conductivity. The real part of the
electrical permittivity (which is associated with lossless behavior) does not directly appear in \eqref{eq:currcorr}. While \eqref{eq:currcorr} gives the entire covariance matrix, only its diagonal entries are needed for our reconstructions. In
general \eqref{eq:currcorr} should use the energy of a quantum
oscillator \cite{Rytov:1989:PSR3} instead of $\kappa T$, namely
\begin{equation}
\Theta(T,\omega) = \frac{\hslash \omega}{2} \coth \frac{\hslash
	\omega}{2\kappa T},
\end{equation}
where $\hslash \approx 1.05 \times 10^{-34} \textrm{J}\cdot\textrm{s}$ is Planck's constant. Here we assume we work with relatively small frequencies so that
$\kappa T \gg \hslash \omega$ and we can make the approximation $\Theta(T,\omega) \approx
\kappa T$, see also \cite{Landau:1968:CTP}. In particular, this approximation is valid at room temperature and frequencies of the order of $1$kHz or $1$MHz.

Instead of working with the Maxwell equations, we use a quasi-static
approximation  that is used in electrical
impedance tomography, see e.g. \cite{Cheney:1999:EIT,Borcea:2002:EIT}. In this approximation, it is convenient to define
the complex conductivity $\sigma$  by
\begin{equation}
\sigma(x) = \sigma'(x) + i \omega\varepsilon'.
\label{eq:cplxsigma}
\end{equation}
If we assume that  $\omega \mu |\sigma| L^2
\ll 1$, where $L$ is the characteristic length of the problem, then one can use the approximation $\nabla \times E \approx
0$. In other words, we may assume that the electric field comes from a potential $E = -
\nabla \phi$. By taking divergence on both sides of the first equation
in \eqref{eq:maxwell} we get
\begin{equation}
\nabla \cdot [ \sigma \nabla \phi] = \nabla \cdot j_e.
\end{equation} 
\begin{remark}
	As noted in \cite{Cheney:1999:EIT}, the quasi-static approximation holds for conductivities consistent with human tissues (see, e.g., \cite{Borcea:2002:EIT}). For example if we take $L = 10~\mathrm{cm}$, $\sigma' = 2~\mathrm{cm}^{-1} \mathrm{k}\Omega^{-1}$,  $\omega = 2\pi10~\mathrm{kHz}$ and $\epsilon' = 1 \mu\mathrm{F}/\mathrm{m}$, we get $\omega \mu |\sigma| L^2 \approx 1.7\times 10^{-4} \ll 1$.
    % using Unitful: m, cm, K, J, Ω, A, kHz, H, μF
	%  4π*1e-7H/m*2π*10kHz*(10cm)^2*abs(1/(500*Ω*cm) + im*2π*10kHz*1μF/m) |> upreferred
    %  0.00016552306075063195
\end{remark}
%%%%%%%%%%%%%%%%%%%%%%%%%%%%%%%%%%%%%%%%%%%%%%%%%%%%%%%%%%%%%%%%%%%%%%%
\section{From the stochastic to the deterministic problem}
\label{sec:det}
Let $\Omega$ be a smooth simply connected open domain of $\real^3$ and let $\sigma \in C^1(\overline{\Omega})$ with its real part satisfying $\sigma'>c$ for some positive constant $c$. We assume a potential $\phi$ satisfies
\begin{equation}
\begin{aligned}
\nabla \cdot [ \sigma \nabla \phi ] &= \nabla \cdot j_e,~\text{in}~\Omega,\\
\phi &= 0,~\text{on}~\partial\Omega.
\end{aligned}
\label{eq:phi}
\end{equation}
Here $j_e$ is the random current term with $\Ma{j_e} = 0$ and $\Ma{|j_e|^2}$ given in \eqref{eq:currcorr}. We assume that $j_e$ is $C^1(\overline{\Omega})$ and hence $\phi$ is $C^2(\overline{\Omega})$ \cite[Ch. 6.3]{Evans:2010:PDE}.

We assume we measure currents flowing out of the domain $\Omega$ at $n$
``electrodes'' by the complex vector with $n$ entries
\begin{equation}
J = \int_{\partial \Omega} dS(x) \begin{bmatrix}
e_1(x)\\\vdots\\e_n(x)
\end{bmatrix}\sigma(x) \nabla \phi(x) \cdot \nu(x),
\end{equation}
where $\nu(x)$ is the unit outward pointing normal to $\partial\Omega$ at
some $x\in \partial\Omega$, and the function $e_i(x)$ are possibly complex $C^1(\partial \Omega)$ ``electrode functions'' defined on $\partial\Omega$. For example
they could be a continuously differentiable approximation of the characteristic function of electrodes at the boundary.

In the following result we prove that the $n \times n$ covariance matrix $\Ma{J
J^*}$ of such measurements can be related to solutions to deterministic auxiliary problems.

\begin{theorem}
	\label{thm:corr}
	The covariance of the vector of measurements $J$ are given by
	\begin{equation}
	[\Ma{JJ^*}]_{ij} = \frac{\kappa}{\pi} \int_\Omega dy\, \Re(\sigma(y)) T(y)
	\nabla u_i(y) \cdot \nabla \overline{u_j(y)} \Delta \omega,
	\end{equation}
	where the functions $u_j$ are solutions to the Dirichlet problems
	\begin{equation}
	\begin{aligned}
	\nabla \cdot [ \sigma \nabla u_i] &= 0,~\text{in}~\Omega\\
	u_i &= e_i,~\text{on}~\partial\Omega.
	\end{aligned}
	\label{eq:aux1}
	\end{equation}
\end{theorem}
\begin{proof}
	First, note that we can write the solution  to \eqref{eq:phi} as
	\begin{equation}
	\phi(x) = \int_\Omega dy\, G(x,y) \nabla_y \cdot j_e(y),
	\label{eq:greensol}
	\end{equation}
	where $G(x,y)$ is the Green function $G(x,y)$ satisfying the equation
	\begin{equation}
	\begin{aligned}
	\nabla_x \cdot [ \sigma(x) \nabla_x G(x,y) ] &= \delta(x-y), ~ x,y \in
	\Omega\\
	G(x,y) &=0,~x\in \partial\Omega ~\text{or}~y\in \partial\Omega.
	\end{aligned}
	\end{equation}
	To double check \eqref{eq:greensol}, it is clear that $\phi(x) = 0$ for $x \in \partial \Omega$
	because $G(x,y) = 0$ for $x \in \partial \Omega$. Also:
	\[
	\begin{aligned}
	\nabla_x \cdot [ \sigma(x) \nabla_x \phi] &= 
	\int_\Omega dy\, \nabla_x \cdot[ \sigma(x) \nabla_x G(x,y) ] \nabla_y
	\cdot j_e(y)\\
	&= \int_\Omega dy\, \delta(x-y) \nabla_y \cdot j_e(y) \\
	&= \nabla_x \cdot  j_e(x).
	\end{aligned}
	\]

	Now it is  helpful to use integration by parts to get
	\begin{equation}\label{eqn:int_by_parts}
	\begin{aligned}
	\phi(x) &= \int_\Omega dy\, G(x,y) \nabla_y \cdot
	j_e(y)\\
	&= \int_{\partial\Omega} dS(y)\, G(x,y) j_e(y) \cdot \nu(y)
	- \int_\Omega  dy\, \nabla_y G(x,y) \cdot j_e(y)\\
	&= -\int_\Omega  dy\, \nabla_y G(x,y) \cdot j_e(y).
	\end{aligned}
	\end{equation}
	
	Using \eqref{eqn:int_by_parts} and \eqref{eq:currcorr} allows us to compute 
	\begin{equation}
	\begin{aligned}
	\Ma{JJ^*}_{ij} &= \Ma{\int_{\partial\Omega}dS(x)\int_{\partial\Omega}dS(x')
		e_i(x) \sigma(x)  \nabla_x \phi(x) \cdot \nu(x)
		\overline{e_j(x')} \overline{\sigma(x')} \nabla_{x'} \overline{\phi(x')}  \cdot
		\nu(x')}\\
	&=\int_{\partial\Omega}dS(x)\int_{\partial\Omega}dS(x')\int_\Omega
	dy \frac{\kappa}{\pi} \Re(\sigma(y)) T(y) \Delta \omega
	e_i(x) \sigma(x) 
	\overline{e_j(x')} \overline{\sigma(x')} 
	\nu(x)^T \\ & \hspace{4.6cm} \nabla_x \nabla_y G(x,y) 
	\nabla_{x'} \nabla_y \overline{G(x',y)} \nu(x')\\
	&=\frac{\kappa}{\pi} \int_\Omega dy\, \Re(\sigma(y)) T(y) \Delta \omega
	\Mb{\int_{\partial\Omega}dS(x) e_i(x) \sigma(x)  \nabla_x \nabla_y
		G(x,y) \nu(x)}^T \\
	& \hspace{4.5cm}\overline{
		\Mb{\int_{\partial\Omega}dS(x') e_j(x') \sigma(x')  \nabla_{x'} \nabla_y
			G(x',y) \nu(x')} }\\
	&= \frac{\kappa}{\pi} \int_\Omega dy\, \Re(\sigma(y)) T(y) \Delta \omega
	\nabla_y u_i(y) \cdot \nabla_y \overline{u_j(y)},
	\end{aligned}
	\end{equation} 
	where $u_i$ solves the problem \eqref{eq:aux1} and we use that 
	$
	\nabla_x \nabla_y G(x,y)
	$
	is symmetric. The last equality follows by doing integration by parts twice:
	\begin{equation}
	\begin{aligned}
	\int_{\partial\Omega}dS(x) e_i(x) \sigma(x)  \nabla_x 
	G(x,y) \cdot \nu(x) 
	&= \int_\Omega dx \nabla_x \cdot [ \sigma(x) \nabla_x G(x,y) u_i(x) ]
	\\
	&= \int_\Omega dx \nabla_x \cdot [ \sigma(x) \nabla_x G(x,y)] u_i(x) 
	\\&+ \int_\Omega dx \sigma(x) \nabla_x G(x,y) \cdot \nabla_x u_i(x)\\
	&= u_i(y) + \int_\Omega dx \nabla_x \cdot[ \sigma (x) G(x,y) \nabla_x
	u_i(x)] \\&- \int_\Omega dx G(x,y) \nabla_x \cdot [ \sigma(x) \nabla_x
	u_i(x)]\\
	&= u_i(y) + \int_{\partial\Omega} dS(x) G(x,y) \sigma(x) \nabla_x
	u_i(x) \cdot \nu(x)\\
	&= u_i(y).
	\end{aligned}
	\end{equation}
\end{proof}
It may be possible to loosen the regularity assumptions on $j_e$ and $e_i$ and derive a similar result to \cref{thm:corr}. Since the scope of this work is focused on establishing the relation between the stochastic and deterministic problems, we leave this for future work. 

%%%%%%%%%%%%%%%%%%%%%%%%%%%%%%%%%%%%%%%%%%%%%%
\subsection{Boundary conditions modeling electrodes with insulating gaps}
The setup using Dirichlet boundary conditions \eqref{eq:phi} assumes that $\phi|_{\partial\Omega} = 0$, which would likely be hard to realize in practice because we expect to have a few electrodes connected to the ground with insulating gaps between them. This corresponds to a boundary condition of mixed type: homogeneous Dirichlet on the electrodes and homogeneous Neumann (zero flux) on the gaps between the electrodes. To be more precise, let $\Gamma = \supp e_1 \cup \ldots \cup \supp e_n$ then we replace \eqref{eq:phi} with 
\begin{equation}\label{eq:mixed}
\begin{aligned}
\nabla \cdot [ \sigma \nabla \phi ] &= \nabla \cdot j_e, ~\text{in}~\Omega,\\
\sigma \nabla \phi \cdot \nu &= 0,~ \text{on} ~ \partial\Omega -
\Gamma, \\
\phi &= 0, ~\text{on}~\Gamma.
\end{aligned}
\end{equation}
Then \cref{thm:corr} holds in the same fashion, but we assume that the $u_i$ are solutions to the following mixed boundary problem replacing \eqref{eq:aux1} with
\begin{equation}
\begin{aligned}
\nabla \cdot [ \sigma \nabla u_i] &= 0,~\text{in}~\Omega,\\
\sigma \nabla u_i \cdot \nu &= 0,~\text{on}~\partial\Omega-\Gamma,\\
u_i &= e_i,~\text{on}~\Gamma.
\end{aligned}
\label{eq:aux2}
\end{equation}
The proof follows by noting that integration by parts now yields 
\[ 
\phi(x) = \int_{\partial \Omega - \Gamma} dS(y) G(x,y) j_e(y) \cdot \nu(y) - \int_\Omega dy \nabla_y G(x,y) \cdot j_e(y),
\]
resulting in four terms when $\phi(x)$ is substituted in $\Ma{JJ^*}_{ij}$. One of the terms is similar to the case of the Dirichlet boundary conditions, and all of the others contain integrals over the zero flux part of the boundary. By invoking the zero flux boundary conditions, these terms can be easily seen to disappear, leaving us with the same formula for $\Ma{JJ^*}_{ij}.$

%%%%%%%%%%%%%%%%%%%%%%%%%%%%%%%%%%%%%%%%%%%%%%%%%%%%%%%%%
\subsection{Differential temperature measurements}
Utilizing \cref{thm:corr}, we can now relate the differential temperature measurements as described in \cref{sec:intro} to measurements of the internal functional \eqref{eqn:func}. Concretely, we take a set of measurements of the covariance of the currents in a body at temperatures $T = T_0$ and $T = T_0+\delta T(x)$, where $\delta T(x)$ is a prescribed heating pattern. Then the differential temperature measurements give
\begin{equation}\label{eqn:diff_temp}
\Mb{ \Ma{J_{T_0 + \delta T} J_{T_0 + \delta T}^*} - \Ma{J_{T_0}
		J_{T_0}^*} }_{ii} = \frac{\kappa}{\pi} \int_\Omega dx\, \delta T (x) \Delta \omega
\Re(\sigma(x)) |\nabla u_i(x)|^2,
\end{equation}
considering only the diagonal elements of the covariance matrix. As previously noted, only measurements of the diagonal elements are used for our reproduction approach. By taking a sufficiently rich set of heating patterns we get an estimate for the internal functional 
\begin{equation}
H_{ii}(x) = \sigma'(x)|\nabla u_i(x)|^2,~\text{for}~x\in\Omega,~i,j\in\{1,\ldots,n\}.
\label{eqn:intfun}
\end{equation}
For real conductivities ($\sigma=\sigma'$), the internal functional \eqref{eqn:intfun} corresponds to the power dissipated inside the domain.

\begin{remark}
	In our numerical experiments (see \cref{sec:numrec}) we use, for convenience, heating patterns $\delta T(x)$ that are approximate Dirac delta distributions. Other patterns such as cosines and sines could also be used. Using spatially extended patterns may be advantageous in terms of signal to noise ratio.
\end{remark}

%%%%%%%%%%%%%%%%%%%%%%%%%%%%%%%%%%%%%%%%%%%%%%%%%%
\subsection{Rough estimation of thermal noise induced currents}
\label{rem:measure_size}
	The thermal noise induced currents are very small and may limit the application of this approach. To get an idea of the magnitude of the signals that need to be measured to obtain $H_{ii}$ in \eqref{eqn:diff_temp}, we need to distinguish between current measurements with the background temperature $T_0$ and with perturbed temperature  $T_0 + \delta T$. We make rough estimates of these currents in two situations: the first is consistent with the numerical experiments and the second one is consistent with laser welding. 

	%%%%%%%%%%%%%%%%%%%%%%%%%%%%%%%%%%%%%%%%%%%%%%%
	\subsubsection*{Conductivities used in the numerical experiments}
	For the background temperature measurements, recall that Boltzmann's constant is on the order of $10^{-23}$ $\textrm{J} \cdot \textrm{K}^{-1}$. For our numerical experiments we chose $\Delta \omega = 10$ kHz, $\Delta z = 0.1$ cm and a domain with area 10 cm$^2$. If the conductivity is about $10^{-3}$ cm$^{-1}$ $\Omega^{-1}$ and $T_0 = 300$ K, then accounting for the $1/\pi$ factor, the variance of the random currents is on the order of $10^{-20}$ A$^2$. To reach this estimate we assumed the squared gradient of the auxiliary fields is constant and equal to $10^{-2}$ cm$^{-2}$. For the differential measurements we may further assume a $\delta  T = 10$ K on area of $(0.2)^2$ cm$^2$. This gives a current variance of the order $10^{-25}$ A$^2$ and a signal to noise ratio of $10^{-5}$. % I guess this is 50dB?

	% using Unitful: m, cm, K, J, Ω, A, kHz
	% 1.36e-23J/K * 300K * 0.1cm * (10cm)^2 * (1/π) * 1e-2cm^-2 * 1e-3/Ω/cm * 10kHz * 2π	 |> A^2
    % 8.160000000000002e-21 A²
	% 1.36e-23J/K * 10K * 0.1cm * (0.2cm)^2 * (1/π) * 1e-2cm^-2 * 1e-3/Ω/cm * 10kHz * 2π |> A^2
    % 1.0880000000000006e-25 A²

	%%%%%%%%%%%%%%%%%%%%%%%%%%%%%%%%%%%%%%%%%%%%%%%
	\subsubsection*{Conductivities consistent with welding}
	The conductivity of gold is much higher than what we used in the numerical experiments and is on the order of $4.5\times 10^7$m$^{-1}$ $\Omega^{-1}$. For instance consider a sheet of gold of dimensions  1 cm $\times$  1 cm $\times$ 1 mm and a bandwidth and central frequencies on the order of 100 Hz. For this choice of frequencies, the quasi-static approximation (\cref{sec:quasi}) is not well satisfied. Nevertheless, if $T_0 = 300$ K  the variance of the random currents is on the order of $10^{-14} \textrm{A}^2$. For the differential measurements we may further assume a $\Delta T = 1300$ K (which is close to the melting point of gold) on an area of $(0.1)^2$ mm$^2$. This gives a current variance of the order $10^{-17}$ A$^2$ and a signal to noise ratio of $10^{-3}$. 
	
	% Running other_cases.jl
	% ----------- Welding --------------
	% Quasitatic approx: 3.553057584392169 (should be small)
	% L = 1 cm, Δx = 0.1 mm, resolution L/Δx = 100.0
	% Δz = 1 mm, Δω = 628.3185307179587 Hz, ∇u2 = 1.0 cm⁻²
	% background current variance = 3.672e-14 A²
	% signal     current variance = 1.5912000000000006e-17 A²

%%%%%%%%%%%%%%%%%%%%%%%%%%%%%%%%%%%%%%%%%%%%%%%%%%%%%%%%%
\subsection{The inverse problem for real conductivities}
The inverse problem for a real conductivity $\sigma=\sigma'$ consists of the measurement equation \eqref{eqn:intfun} and the auxiliary problem \eqref{eq:aux1}. To be more precise, we seek to recover $u_i$ and $\sigma$ given $H_{ii}$ and $e_i$ from the real non-linear system of partial differential equations, for $i=1,...,n$,
\begin{equation}\label{eqn:system_cont}
\begin{aligned}
\nabla \cdot \Mb{\sigma \nabla u_i} &= 0,  \quad  x\in \Omega,\\ 
u_i-e_i &= 0,  \quad  x\in \partial \Omega,\\ 
H_{ii} - \sigma |\nabla u_i|^2 &= 0,  \quad x\in \Omega.
\end{aligned}
\end{equation}
We call the model associated with measurements $H_{ii}$ given by the expectation in \cref{thm:corr} the \textit{deterministic model} and the model associated with measurements given by realizations of randomly induced currents the \textit{stochastic model}. 

%%%%%%%%%%%%%%%%%%%%%%%%%%%%%%%%%%%%%%%%%%%%%%%%%%%%%%%%%
\subsection{The inverse problem for complex conductivities}
We write the complex problem by separating the real and complex parts of \eqref{eqn:intfun} and \eqref{eq:aux1}. To avoid confusion with the complex number $i$ we use $j$ to denote experiments for the complex conductivity problem. We use a single prime (resp. double prime) to denote the real part (resp. imaginary part) of a complex quantity, e.g. $\sigma = \sigma'+i\sigma''$, $u_j = u_j'+iu_j''$, and $e_j=e_j'+ie_j''$. Then the problem is to find $\sigma'$, $\sigma''$, $u_j'$ and $u_j''$ given $H_{jj}$, $e_j'$, and $e_j''$ from the non-linear system of partial differential equations, for $j=1,..,n$, 
\begin{equation}\label{eqn:system_cont_complex}
\begin{aligned}
\nabla \cdot \Mb{\sigma '\nabla u_j'}-\nabla \cdot \Mb{\sigma ''\nabla u_j''} &= 0,  \quad x\in \Omega,\\ 
\nabla \cdot \Mb{\sigma '\nabla u_j''}+\nabla \cdot \Mb{\sigma ''\nabla u_j'} &= 0,  \quad x\in \Omega,\\ 
u_j'-e_j'&=0, \quad  x\in \partial \Omega, \\
u_j''-e_j''&=0, \quad  x\in \partial \Omega, \\
H_{jj}-\sigma'(|\nabla u_j'|^2+|\nabla u_j''|^2) &= 0, \quad x\in\Omega.
\end{aligned}
\end{equation}
An equivalent formulation of \eqref{eqn:system_cont_complex} can be found using the conjugates of $u_j$ and the $e_j$ instead of their real and imaginary components separately. Both the system \eqref{eqn:system_cont} and \eqref{eqn:system_cont_complex} can be modified to instead use the experimental boundary conditions \eqref{eq:aux2}. 

\begin{remark}
The non-linear system of equations that would be obtained by allowing the conductivity to be complex in ultrasound modulated EIT (see e.g. \cite{Bal:2013:HIP}) is similar to \eqref{eqn:system_cont_complex} with two real measurement equations per boundary condition instead of a single one, i.e. for $x\in\Omega$:
\[
	\begin{aligned}
	H_{jj}'-\sigma'(|\nabla u_j'|^2+|\nabla u_j''|^2) &= 0,~\text{and}\\
	H_{jj}''-\sigma''(|\nabla u_j'|^2+|\nabla u_j''|^2) &= 0.
	\end{aligned}
\]
We did not consider this problem because the form of the measurements we consider \eqref{eqn:system_cont_complex} is a direct result of using thermal induced random currents, see \cref{thm:corr}.
\end{remark}

%%%%%%%%%%%%%%%%%%%%%%%%%%%%%%%%%%%%%%%%%%%%%%%%%%%%%%%%%%%%%%%%%%%%%%%
\section{Linearized problem}\label{sec:stability}
Before attempting to reconstruct conductivities numerically, we analyze the linearizations of the real \eqref{eqn:system_cont} and complex \eqref{eqn:system_cont_complex} conductivity problems. Our goal is to find sufficient conditions for injectivity of the linearized problems, or in other words, if they admit a unique solution. Our analysis is based on \cite{Bal:2014:HIP}, which includes a proof that the linearized real conductivity problem is elliptic in the sense of Douglis-Nirenberg under certain boundary conditions \cite{Douglis:1955:IEE}. This was established in \cite{Kuchment:2012:SIP} for ultrasound modulated EIT and generalized to other hybrid inverse problem in \cite{Bal:2014:HIP}.

The linearization of the real conductivity problem \eqref{eqn:system_cont} around the solution $(u_i,\sigma)$ in the variables $(\delta u_i, \delta \sigma)$ for $i=1,..,n$ is given by 
\begin{equation}\label{eqn:linreal}
\begin{aligned}
\nabla \cdot \Mb{\sigma \nabla \delta u_i} + \nabla \cdot \Mb{\delta \sigma \nabla u_i }&= 0,  \quad  x\in \Omega,\\ 
\delta u_i &= 0,  \quad x\in \partial \Omega,\\ 
\delta H_{ii} - \delta \sigma|\nabla u_i|^2-2\sigma\nabla \delta u_i \cdot \nabla u_i &= 0,  \quad  x\in \Omega.
\end{aligned}
\end{equation}
The linearization of the complex conductivity problem \eqref{eqn:system_cont_complex} around the solution $(u_j',u_j'',\sigma',\sigma'')$ in the variables $(\delta u_j',\delta u_j'', \delta \sigma', \delta \sigma'')$ for $j=1,..,n$ is given by 
\begin{equation}
\begin{aligned}\label{eqn:lincom}
\nabla \cdot \Mb{\sigma' \nabla \delta u_j'} + \nabla \cdot \Mb{\delta \sigma' \nabla u_j' } 
-\nabla \cdot \Mb{\sigma'' \nabla \delta u_j''} - \nabla \cdot \Mb{\delta \sigma'' \nabla u_j'' }&= 0,  \quad  x\in \Omega,\\ 
\nabla \cdot \Mb{\sigma' \nabla \delta u_j''} + \nabla \cdot \Mb{\delta \sigma' \nabla u_j'' } 
+\nabla \cdot \Mb{\sigma'' \nabla \delta u_j'} + \nabla \cdot \Mb{\delta \sigma'' \nabla u_j' }&= 0 , \quad  x\in \Omega,\\ 
\delta u_j' &= 0,  \quad  x\in \partial \Omega,\\ 
\delta u_j'' &= 0,  \quad  x\in \partial \Omega,\\ 
\delta H_{jj} - \delta \sigma'(|\nabla u_j'|^2+|\nabla u_j''|^2)-2\sigma'\nabla \delta u_j' \cdot \nabla u_j'-2\sigma'\nabla \delta u_j' \cdot \nabla u_j'&=0,\quad x \in \Omega.
\end{aligned}
\end{equation}

In \cite{Bal:2014:HIP}, it is established \eqref{eqn:linreal} is elliptic using two boundary conditions if the gradients of the associated fields are nowhere orthogonal or parallel. We do not attempt to analyze how this condition might be satisfied in the case of the mixed boundary conditions \eqref{eq:aux2}. We note that this establishes that \eqref{eqn:linreal} is not elliptic in the case of one experiment with mixed boundary conditions. Instead, we attempt to analyze the problem numerically by estimating the conditioning of the symbol of the linearized problem. For the case of complex conductivity \eqref{eqn:lincom}, we give a sufficient condition in \cref{lem:sym} for ellipticity; however, we do not give boundary conditions that guarantee this is satisfied, nor do we prove that such boundary conditions exist. We note for elliptic linear systems, it is possible to obtain stability estimates by augmenting the  system with boundary conditions satisfying the Lopatinskii condition, following \cite{Bal:2014:HIP}.

To establish if \eqref{eqn:linreal} and \eqref{eqn:lincom} are elliptic, we first compute the principal symbol of their associated matrix-valued differential operators  $\mathcal{A}(x,D)$ for $x \in \Omega$, where $D=(\partial_{x_1},...,\partial_{x_n})$. Since these are linearized systems, the entries $\mathcal{A}_{ij}(x,D)$ are polynomials in $D$ for each $x \in \Omega$. We associate each row of $\mathcal{A}$ with an integer $s_i$ and each column with an integer $t_j$, chosen such that the maximum degree of each  polynomial $\mathcal{A}_{ij}(x,D)$ is $s_i+t_j$. The principal component $\mathcal{A}_0(x,D)$ is obtained from $\mathcal{A}(x,D)$ by keeping only the terms in $\mathcal{A}_{ij}(x,D)$ with order exactly $s_i+t_j$. If the principal symbol $\mathcal{A}_0(x,\xi)$ is injective for all $\xi \neq 0$, then the problem is elliptic in the Douglis-Nirenberg sense at $x\in \Omega$. 

%%%%%%%%%%%%%%%%%%%%%%%%%%%%%%%%%%%%%%%%%%%%%%%%%%%%%%%%%%%%%%%%%%%%%%%
\subsection{Injectivity of the linearized real problem} \label{sec:real_prob}
 Letting $F_i =\nabla u_i$ the principal symbol of the real problem \eqref{eqn:linreal} is the $2n \times (n+1)$ matrix 
\begin{equation}\label{eqn:real_sym}
\mathcal{A}_0(x,\xi) = \begin{bmatrix}
|F_1|^2 & 2 \sigma F_1 \cdot i \xi &  \cdots  & 0\\
F_1 \cdot i \xi  & -\sigma |\xi|^2 &   \cdots  & 0 \\
\vdots & \vdots &  \ddots &  \vdots \\
|F_n|^2  & 0 & \cdots & 2\sigma F_n \cdot i \xi\\
F_n \cdot i\xi  & 0 & \cdots & -\sigma |\xi|^2
\end{bmatrix},
\end{equation}
for $i =1,...,n$ and where $\sigma = \sigma'$. The system is in Douglis-Nirenberg form where the row weights $s_i$ are given by the $2n$ vector $(0,1,0,1,\ldots,0,1)$ and the column weights $t_j$ are given by the $n+1$ vector $(0,1,1,\ldots,1)$. As noted previously, this symbol is shown to be injective in two dimensions using two boundary conditions such that $F_1$ and $F_2$ are nowhere orthogonal, or parallel \cite{Bal:2014:HIP}.  

We consider the discretized problem on the square $[0,10]^2$ using a uniform $200 \times 200$ grid. The conductivity used can be seen in \cref{fig:det_real} (a). We numerically solve \eqref{eq:aux2} to calculate the fields $u_i$ for $i=1,..,n$. The Dirichlet boundary conditions are defined on the set 
\begin{equation}\label{eqn:gamma}
\begin{aligned}
\Gamma = \big(([0,4.5] \cup [5.5,10]) \times \{0,10\}\big) \cup \big(\{0,10\} \times ([0,4.5] \cup [5.5,10])\big).
\end{aligned}
\end{equation}
For $x\in \Gamma$ the boundary conditions are of the form
\begin{equation}\label{eqn:bc}
\begin{aligned}
g_n &= 5\sin(\theta)\M{\frac{r}{10}}^n,\\
h_n &= 5\cos(\theta)\M{\frac{r}{10}}^n,
\end{aligned}
\end{equation}
where $(r,\theta)$ is the polar representation of the nodes in $\Gamma$. We enforce the no flux boundary condition on the gaps, i.e. $x \in \partial \Omega - \Gamma$. 

To establish the ellipticity of the operator, it is sufficient to show that it is injective for all $\xi$ such that $|\xi|=1$. We check this condition numerically for $\xi \in \Xi$, where $\Xi$ is a set of 100 vectors uniformly spaced on the unit circle (since $\xi$ is two-dimensional in our simulations). At each point $x$ in the grid we use to discretize $\Omega$, we compute the maximum condition number of the symbol along directions $\xi \in \Xi$, i.e.
\begin{equation}\label{eqn:condition}
\max_{\xi \in \Xi}\Mb{\frac{\sigma_{\max}(\mathcal{A}_0(x,\xi))}{\sigma_{\min}(\mathcal{A}_0(x,\xi))}},
\end{equation}
where $\sigma_{\min}(A)$ (resp. $\sigma_{\max}(A)$) is the smallest (resp. largest) singular value of a matrix $A$.
The maximum condition number \eqref{eqn:condition} of the symbol \eqref{eqn:real_sym} can be seen in \cref{fig:cond_gap}. These numerical results are in line with the previously established theory: the maximum condition number is higher under one experiment than under two experiments. Even under one experiment, the conditioning is still reasonable for most applications. Indeed, the conductivity can still be ``reasonably'' recovered using one experiment but with worsened numerical artifacts. 

\begin{figure}[h!]
	\centering
	\begin{tabular}{cc}
			\includegraphics[width=55mm]{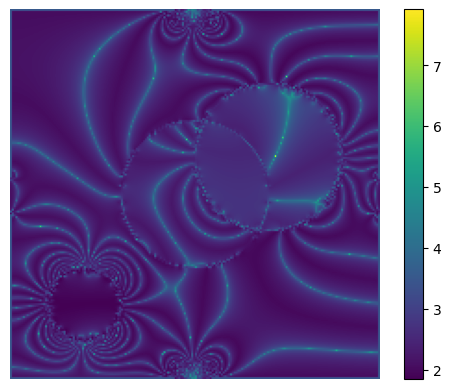} & \includegraphics[width=55mm]{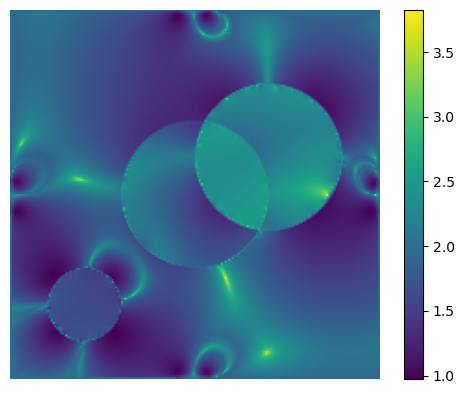} \\
			(a) One boundary condition & (b) Two boundary conditions
 	\end{tabular}
	\caption{Maximum condition number \eqref{eqn:condition} of the symbol of the linearized problem \eqref{eqn:linreal} on the square domain $[0,10]^2$ with $\log10$ scaling. The left image (a) is the conditioning of the symbol with one boundary condition given by $g_1$ in \eqref{eqn:bc}. The right image (b) is the conditioning of the symbol with two boundary conditions given by $g_1$ and $h_1$ in \eqref{eqn:bc}.} \label{fig:cond_gap}
\end{figure}

\subsection{Injectivity of the linearized complex problem}\label{sec:comp_prob} Letting $F_j' = \nabla u_j'$, $F_j'' = \nabla u_j''$, and $F_j = F_j'+iF_j''$, the symbol for the complex linear system \eqref{eqn:lincom} is the $3n\times (2+2n)$ matrix 
\begin{equation}\label{eqn:comp_sym}
\widetilde{\mathcal{A}}_0(x,\xi) = \begin{bmatrix}
|F_1|^2 & 0 & 2 \sigma' F_1' \cdot i \xi &  2\sigma' F_1'' \cdot i \xi & \cdots & 0 &0\\
F_1' \cdot i \xi  & -F_1'' \cdot i\xi &-\sigma'|\xi|^2 & \sigma''|\xi|^2 &\cdots & 0 & 0 \\
F_1'' \cdot i \xi  & F_1' \cdot i\xi &-\sigma''|\xi|^2 & -\sigma'|\xi|^2 &\cdots & 0 & 0 \\
\vdots & \vdots & \vdots & \ddots & \ddots & \vdots & \vdots \\
|F_n|^2 & 0 & \cdots & 0 & \cdots &  2 \sigma' F_n' \cdot i \xi &  2\sigma' F_n'' \cdot i \xi\\
F_n' \cdot i \xi  & -F_n'' \cdot i\xi &\cdots & 0 & \cdots & -\sigma'|\xi|^2 & \sigma''|\xi|^2 \\
F_n'' \cdot i \xi  & F_n' \cdot i\xi &\cdots & 0 & \cdots &-\sigma''|\xi|^2 & -\sigma'|\xi|^2
\end{bmatrix},
\end{equation}
for $j=1,...,n$. The system is in Douglis-Nirenberg form where the row weights $s_i$ are given by the $3n$ vector $(0,1,1,0,1,1,\ldots, 0,1,1)$ and the column weights $t_j$ are given by the $2+2n$ vector $(0,0,1,1,\ldots,1,1)$.

\begin{lemma}\label{lem:sym}
	The symbol of the system \eqref{eqn:lincom}, $\widetilde{\mathcal{A}}_0(x,\xi)$, with $n \geq 2$ is injective at $x\in \Omega$ if there exists two experiments $i$ and $j$ such that for all $\xi \neq 0$,
	\begin{equation}\label{eqn:cond_det}
	|F_i|^2 |F_j \cdot \xi|^2 \neq 
	|F_j|^2 |F_i \cdot \xi|^2,
	\end{equation}
	and $\sigma'(x),\sigma''(x)>0$. 
\end{lemma}
\begin{proof}
	If $n=2$, we write the system such that the measurement equations are the first two rows 
\[
	\begin{bmatrix}
|F_1|^2 & 0 & 2 \sigma' F_1' \cdot i \xi &  2\sigma' F_1'' \cdot i \xi  & 0 &0\\
|F_2|^2 & 0 & 0 & 0 &  2 \sigma' F_2' \cdot i \xi &  2\sigma' F_2'' \cdot i \xi\\
F_1' \cdot i \xi  & -F_1'' \cdot i\xi &-\sigma'|\xi|^2 & \sigma''|\xi|^2  & 0 & 0 \\
F_1'' \cdot i \xi  & F_1' \cdot i\xi & -\sigma''|\xi|^2 & -\sigma'|\xi|^2 & 0 & 0 \\
F_2' \cdot i \xi  & -F_2'' \cdot i\xi &0 & 0 & -\sigma'|\xi|^2 & \sigma''|\xi|^2 \\
F_2'' \cdot i \xi  & F_2' \cdot i\xi &0 & 0 & -\sigma''|\xi|^2 & -\sigma'|\xi|^2
\end{bmatrix}.
\]
We consider this as a block matrix with the top left block being $2 \times 2$, and then the bottom right block being $4 \times 4$. The bottom right matrix is block diagonal with invertible diagonal $2\times 2$ blocks, and we use this to compute the Schur complement 
\[
\begin{bmatrix}
|F_1|^2+\frac{2(\sigma')^2}{(|\xi|\sigma')^2+(|\xi|\sigma'')^2}\M{(F_1' \cdot i \xi)^2 + (F_1'' \cdot i \xi)^2} & \frac{2\sigma''\sigma'}{(|\xi|\sigma')^2 + (|\xi|\sigma'')^2}\M{(F_1' \cdot i \xi)^2 + (F_1'' \cdot i \xi)^2} \\ |F_2|^2+\frac{2(\sigma')^2}{(|\xi|\sigma')^2+(|\xi|\sigma'')^2}\M{(F_2' \cdot i \xi)^2 + (F_2'' \cdot i \xi)^2} & \frac{2\sigma''\sigma'}{(|\xi|\sigma')^2 + (|\xi|\sigma'')^2}\M{(F_2' \cdot i \xi)^2 + (F_2'' \cdot i \xi)^2} 
\end{bmatrix}.
\]
The determinant of the Schur complement is then 
\begin{equation}\label{eqn:schur_det}
\frac{2\sigma''\sigma'}{(\sigma')^2 + (\sigma'')^2} \M{\frac{|F_1|^2}{|\xi|^2}\Mb{(F_2' \cdot i \xi)^2 + (F_2'' \cdot i \xi)^2} -
\frac{|F_2|^2}{|\xi|^2}\Mb{(F_1' \cdot i \xi)^2 + (F_1'' \cdot i \xi)^2}},
\end{equation}
which gives the desired result. 

For $n>2$, without lost of generality we let $i=1$ and $j=2$ in \eqref{eqn:cond_det}. We proceed by using the same approach considering the Schur complement of the largest square sub-matrix obtained by deleting the measurement terms of all experiments $i>2$, that is, the $(2+2n) \times (2+2n)$ matrix,
\[
\begin{bmatrix}
|F_1|^2 & 0 & 2 \sigma' F_1' \cdot i \xi &  2\sigma' F_1'' \cdot i \xi  & 0 &0 & \cdots &0 & 0\\
|F_2|^2 & 0 & 0 & 0 &  2 \sigma' F_2' \cdot i \xi &  2\sigma' F_2'' \cdot i \xi &\cdots &0 & 0\\
F_1' \cdot i \xi  & -F_1'' \cdot i\xi &-\sigma'|\xi|^2 & \sigma''|\xi|^2  & 0 & 0 & \cdots& 0 & 0\\
F_1'' \cdot i \xi  & F_1' \cdot i\xi & -\sigma''|\xi|^2 & -\sigma'|\xi|^2 & 0 & 0 & \cdots& 0& 0\\
F_2' \cdot i \xi  & -F_2'' \cdot i\xi &0 & 0 & -\sigma'|\xi|^2 & \sigma''|\xi|^2 & \cdots &0 &0\\
F_2'' \cdot i \xi  & F_2' \cdot i\xi &0 & 0 & -\sigma''|\xi|^2 & -\sigma'|\xi|^2 & \cdots &0 &0\\
\vdots & \vdots & \vdots & \vdots &\ddots & \ddots & \ddots & \vdots& \vdots\\
F_n' \cdot i \xi  & -F_n'' \cdot i\xi &0 & 0 & 0 &0 & \cdots &  -\sigma'|\xi|^2 & \sigma''|\xi|^2 \\
F_n'' \cdot i \xi  & F_n' \cdot i\xi &0 & 0 & 0 &0 & \cdots &  -\sigma''|\xi|^2 & -\sigma'|\xi|^2 
\end{bmatrix}.
\]
We consider the top left $2 \times 2$ matrix as a block, and the bottom right $2n \times 2n$ matrix as a block. The bottom right matrix is still block diagonal with invertible $2\times 2$ blocks. The determinant of the Schur complement is given by \eqref{eqn:schur_det}, and the determinant of the sub-matrix is given by 
\begin{equation}\label{eqn:full_det}
2\sigma''\sigma'|\xi|^{4n-2}|\sigma|^{2(n-1)}\M{|F_1|^2|F_2 \cdot \xi|^2 -|F_2|^2|F_1 \cdot \xi|^2},
\end{equation}
so by supposition, the sub-matrix is invertible. Since one of the largest square sub-matrices of the symbol is invertible, the symbol is injective. 
\end{proof}

We do not present a method for finding boundary conditions such that the condition in \cref{lem:sym} is satisfied, nor do we know if such boundary conditions exist for all possible $\sigma'(x),\sigma''(x)>0$. If $F_i$ and $F_j$ are unit length then the condition here resembles that in \cite{Bal:2014:HIP}, giving that the fields cannot be orthogonal or parallel. We note the following corollary, which is limited to the case where $n=2$. 
\begin{corollary}\label{cor:zero}
		The symbol of the system \eqref{eqn:lincom}, $\widetilde{\mathcal{A}}_0(x,\xi)$, with $n=2$ is not injective at $x\in \Omega$ if $\sigma'(x) = 0$ or $\sigma''(x)=0$.
\end{corollary}
\begin{proof}
	This follows by observing that the determinant of the Schur complement \eqref{eqn:schur_det} is zero if $\sigma'(x)=0$ or $\sigma''(x)=0$. 
\end{proof}
Showing that a system is not elliptic for all $n$ using determinants is challenging because we need to ensure that the determinants of all maximal square sub-matrices are zero. However this approach is sufficient for the case $n=2$ and we expect \cref{cor:zero} to hold with more experiments. 

Numerically, we have found that the complex reconstruction is challenging under many combinations of boundary conditions. This can be expected from numerically computing the maximum condition number \eqref{eqn:condition} as we illustrate in the numerical experiment appearing in \cref{fig:sym_complex} and that we describe next.

\begin{figure}[h!]
	\centering 
	\begin{tabular}{ccc}
		\includegraphics[width=45mm]{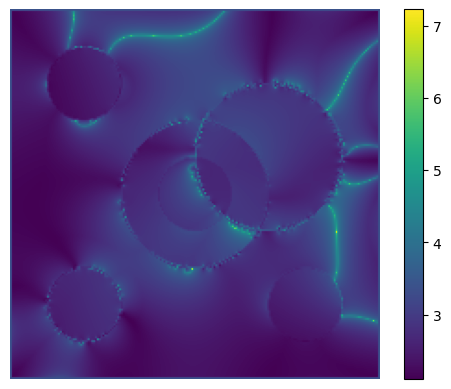} &
		\includegraphics[width=45mm]{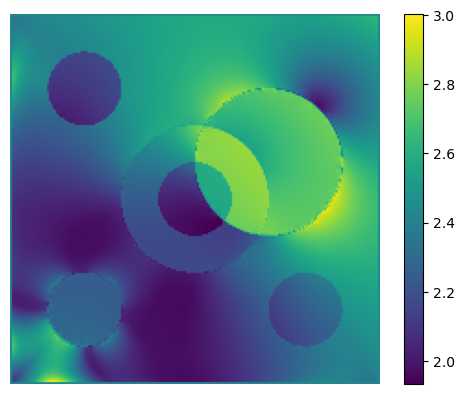} & \includegraphics[width=45mm]{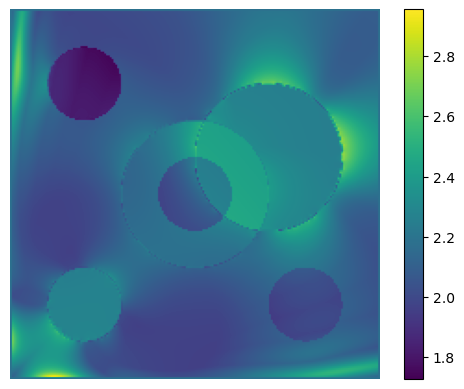} \\
		(a) Two measurements  & (b) Three measurements & (c) Four measurements 
	\end{tabular}
	\caption{The maximum condition number \eqref{eqn:condition} of the symbol of the complex linearized problem \eqref{eqn:lincom} on the discretized square domain with $\log10$ scaling. The images from left to right are given by two, three, and four boundary conditions in the form of \cref{eqn:bc_com}.}
	\label{fig:sym_complex}
\end{figure}

The ground truth conductivity can be seen in \cref{fig:complex_rep} with the real part in (a) and the imaginary part in (b). We only consider the linearized complex problem \eqref{eqn:lincom} with Dirichlet boundary conditions. The boundary conditions are of the form 
\begin{equation}\label{eqn:bc_com}
\begin{aligned}
\tilde{g}_n &= g_n + \frac{i}{2}h_n,\\
\tilde{h}_n &=  h_n + \frac{i}{2}g_n,
\end{aligned}
\end{equation}
with $h$ and $g$ defined in \eqref{eqn:bc}. The scaling of the imaginary part by $1/2$ is to match the imaginary part of the background conductivity. In \cref{fig:sym_complex}, we can see the numerical condition of the symbol for two, three, and four boundary conditions. We begin this experiment with $n=2$ since, with one measurement, the system is underdetermined. The boundary conditions for $n=2$ are $\tilde{g}_1, \tilde{h}_1$, for $n=3$ are  $\tilde{g}_{1}, \tilde{g}_{2}, \tilde{h}_1$, and for $n=4$ are  $\tilde{g}_{1}, \tilde{g}_{2}, \tilde{h}_{1}, \tilde{h}_{2}.$ We use the same domain and grid as in the real case.

The maximum condition number \eqref{eqn:condition} improves significantly by moving from two to three measurements, but the improvement from three to four is modest. The areas with high condition number for two boundary conditions (\cref{fig:sym_complex} (a)) suggest the problem is not elliptic. With more boundary conditions (\cref{fig:sym_complex} (b) and (c)), the areas where the conditioning is high match up with the reconstruction artifacts for the complex case (\cref{fig:complex_rep}).

%%%%%%%%%%%%%%%%%%%%%%%%%%%%%%%%%%%%%%%%%%%%%%%%%%%%%%%%%%%%%%%%%%%%%%%
\section{Numerical reconstructions}
\label{sec:numrec}
The following numerical reconstructions use values consistent with the quasi-static approximation, i.e., values such that $\omega \mu |\sigma|L^2 \ll 1$. In particular we let $L = 10$ cm and $\sigma' \in [1/3,2]$ cm$^{-1}$ $k\Omega^{-1}$. In the case of non-zero complex conductivity we let $\omega = 2\pi10$ kHz and $\sigma'' = \omega \epsilon' \in [1/2,1]$ cm$^{-1}$ $k\Omega^{-1}$. Our choice of parameters is near those in human tissues and satisfies the quasi-static approximation, see e.g. \cite{Borcea:2002:EIT}.
The examples we consider assume a thin plate that is homogeneous in the $z$ direction with thickness $\Delta z = 0.1$cm. If we consider $\Omega \subset \real^2$, then multiplying the measurements by $\Delta z$ corresponds to the results in \cref{sec:det}. 

%%%%%%%%%%%%%%%%%%%%%%%%%%%%%%%%%%%%%%%%%%%%%%%%%%%%%%%%%%%%%%%%%%%%%%%
\subsection{Discrete model}\label{sec:disc_model}
We discretize the system \eqref{eqn:system_cont} on a square domain $\Omega = [0,10]^2$ using a uniform grid with $n^2$ nodes. We denote by $N$ the set of nodes  indexed with their integer coordinates $(i,j)$ and the set of edges by $E \subset N\times N$. The nodes are partitioned into interior nodes $I$ and boundary nodes $B$, which are the nodes that are on the boundary $\partial \Omega$. We use the forward difference operator $D \in \real^{|N| \times |E|}$ defined such that
\[
D = \begin{bmatrix}
D_1 \\ D_2
\end{bmatrix}.
\]
Here $D_1$ (resp. $D_2$) is the horizontal (resp. vertical) first order difference operator. Given a function $\psi$ defined on the nodes $N$, the horizontal and vertical difference operators are defined by
\[
\begin{aligned}
(D_1\psi)(i,j) &= \frac{\psi(i+1,j)-\psi(i,j)}{\Delta x}, \\
(D_2\psi)(i,j) &= \frac{\psi(i,j+1)-\psi(i,j)}{\Delta y},
\end{aligned}
\]
where $\Delta x$ and $\Delta y$ are the horizontal and vertical discretization steps respectively.   

If we use finite differences to discretize \eqref{eqn:system_cont}, we note that the gradient components in the $x$ and $y$ directions are defined on horizontal and vertical edges. Thus the norm of the discretized gradient is not defined at any particular edge. To obtain the gradient at a single spatial location in the discretized problem, we interpolate the gradient approximated values from their respective edges to the nodes and compute the gradient norm at the nodes. Thus it also makes sense to interpret the internal functional $H_{ii}$ as a nodal based quantity and to completely determine the conductivity by interpolating a node based quantity. These interpolations between edges and nodes are achieved with the following matrices
\[
\begin{aligned}
N_1:& \text{horizontal edges} \to \text{nodes}, \\
N_2:& \text{vertical edges} \to \text{nodes}, \\
E_1:& \text{nodes} \to \text{horizontal edges}, \\
E_2:& \text{nodes} \to \text{vertical edges}, \\
E_{1,2}:& \text{nodes} \to \text{all edges}.
\end{aligned}
\]
To define these matrices we let $\phi$ be the matrix-valued function
\[
\phi(A) = (\text{Diag}(|A^T| \mathbf{1} ))^{-1}|A^T|,
\]
where $|\cdot|$ is the entry-wise absolute value, $\mathbf{1}$ is an appropriately sized vector of ones, and $\text{Diag}(v)$ denotes the matrix with the vector $v$ on its diagonal. The matrix $\phi(A)$ preserves constant vectors, more precisely, if $c$ is an appropriate sized constant vector $\phi(A)c = c$. The interpolation operators are then defined as 
\[
N_1 = \phi(D_1^T), \quad N_2 = \phi(D_2^T), \quad E_1 = \phi(D_1), \quad E_2 = \phi(D_2), \quad E_{1,2}= \begin{bmatrix}
E_1 \\ E_2
\end{bmatrix}.
\] 

Given $H_{ii}$ (defined at the nodes) and $e_i$ (defined at the boundary nodes)  the discrete inverse problem for real conductivity is then to find $s$  and $u_i$ (defined at the nodes) such that for $i=1,...,N$, 
\begin{equation}\label{eqn:syst_disc_real}
\begin{aligned}
&D^T[E_{1,2}s \odot (Du_i)]_I = 0,\\
&u_i|_B-e_i =0,  \\
&H_{ii} - \Mb{N_1(E_1 s \odot |D_1u_i|^2)+N_2(E_2 s \odot |D_2u_i|^2)}_I =0,
\end{aligned}
\end{equation}
where $\odot$ is the Hadamard or componentwise product.
We note that in the first equation of \eqref{eqn:syst_disc_real}, we have a graph Laplacian with edge weights given by $E_{1,2}s$, see e.g. \cite{Chung:1997:SGT}.
This system is modified slightly under the assumption that the conductivity is known in a small neighborhood of the boundary. The modified system is solved using Gauss-Newton iteration. However, the interpolation process introduces a null space into the Jacobian. We use Tikhonov regularization with a parameter $\gamma$ to prevent this null space from interfering when solving for the Gauss-Newton step. The parameter corresponds to adding a penalty term of $\gamma \| w \|^2$, when solving the least squares problem for finding a Gauss-Newton step $w$. An Armijo line search is used as  globalization strategy (see e.g. \cite{Nocedal:2006:NO}).

Given $H_{jj}$ and $e_j$, the complex inverse problem is to recover $s'$, $s''$, $u_j'$, and $u_j''$ for $j=1,...,N$,
\begin{equation}\label{eqn:syst_disc_complex}
\begin{aligned}
&D^T[E_{1,2} s' \odot (Du_j')]_I - D^T[E_{1,2} s'' \odot (Du_j'')]_I= 0, \\
&D^T[E_{1,2} s' \odot (Du_j'')]_I + D^T[E_{1,2} s'' \odot (Du_j')]_I= 0,  \\
&u_j'|_B-e_j' =0,  \\
&u_j''|_B-e_j'' =0,  \\
&H_{jj} - \Mb{N_1(E_1 s' \odot |D_1u_j'|^2)+N_2(E_2 s \odot |D_2u_j'|^2)}_I\\&+\Mb{N_1(E_1 s' \odot |D_1u_j''|^2)+N_2(E_2 s' \odot |D_2u_j''|^2)}_I =0.
\end{aligned}
\end{equation}
A similar Gauss-Newton procedure was used to solve \eqref{eqn:syst_disc_complex}.

\textbf{Heating Patterns:} Recall that to obtain the measurements $H_{ii}$; we need to locally heat a region of the conducting plate. This requires numerically approximating \eqref{eqn:diff_temp} at both the background temperature $T_0$ and when the plate is locally heated according to $\delta T(x)$. For our measurements we let $\delta T(x)$ be the Gaussian heating pattern,
\[ 
g( x,a) = (2\pi a)^{-1}\exp\M{-|x|^2(2a)^{-1}}
\]
where $|\cdot|$ denotes the 2-norm. The heating pattern $g(x,a)$ can be considered an approximate Dirac since it integrates in $x$ to one. This is also similar to the heating pattern from a laser covering an area of roughly $\pi a$.

\textbf{Deterministic measurements:} Continuous measurements from the deterministic model using this heating pattern can be written as
\begin{equation}\label{eqn:det_meas}
H_{ii}(x) = \Ma{\sigma |\nabla u_i|^2,T_0+g(\cdot -x, a)}_{L_2(\Omega )}- \Ma{\sigma |\nabla u_i|^2,T_0}_{L_2(\Omega )}.
\end{equation}
We approximate $H_{ii}(x)$ by evaluating the heating pattern at each node $x$ in the discrete model. Then to approximate the inner products, we use a uniform fine grid with $\tilde{n}^2$ nodes such that $\tilde{n}>n$. The number of fine grid nodes $\tilde{n}$ is chosen such that there are at least four fine grid nodes per effective area of the heating pattern, i.e. $\pi a$. 

\textbf{Stochastic measurements:} The measurements from the simulated random current model are given by approximating 
\begin{equation}\label{eqn:stoch_meas}
H_{ii}(x) = \Ma{\Ma{\sigma |\nabla u_i|^2,T_0+g(\cdot -x, a)}_{L_2(\Omega )}} - \Ma{\Ma{\sigma |\nabla u_i|^2,T_0}_{L_2(\Omega )}},
\end{equation}
where the outer angular brackets denote ensemble averaging. The average is approximated empirically with $M$ realizations of random currents where the inner product for each realization is approximated using a uniform fine grid. The realizations of the background temperature measurements (the rightmost term in \eqref{eqn:stoch_meas}) are not recalculated for each heating pattern. For each realization, the currents at every \textit{fine grid} edge midpoint $e$ are sampled from a mean zero random normal distribution with standard deviation $\sqrt{\kappa(T_0+g(e,a))s(e)/\pi}$ (heated) or $\sqrt{\kappa T_0s(e)/\pi}$ (unheated) as determined by \eqref{eq:currcorr}.

We believe the simulation method used for realizations of random currents creates a more challenging problem than experimental data. In practice, measurements could be taken over a time interval and then averaged over time. This gives temporal structure to the data that is not reflected by our simulations which ignore the temporal correlation of the random currents.

\subsection{Real conductivity}\label{sec:real} First, we consider problems of purely real conductivity ($\sigma = \sigma'$) in both the case of measurements from simulated random currents (stochastic model) and measurements using their variances (deterministic model). For both problems we consider the Dirichlet boundary conditions \eqref{eq:aux1} using $e_1 = (x_1+x_2)/10$ and $e_2 = (1+x_1-x_2)/10$. For the more challenging problem of the experimental boundary conditions \eqref{eq:aux2}, we only consider the deterministic model of measurements.  

We show in \cref{fig:det_real} numerical reconstructions for a purely real conductivity in the deterministic and stochastic models. The deterministic measurements $H_{ii}$ in \eqref{eqn:det_meas} are taken with $T_0=300$ and $a = 0.01$ on a $60 \times 60$ coarse grid using a $120 \times 120$ fine grid to approximate the integrals. The same conditions were used for the stochastic model, where in addition we took $T_0=0.01$ using 1000 realizations of random currents. We chose a particularly low background temperature $T_0$ to get clean enough data with the number of realizations we chose. As can be expected from \cref{rem:measure_size}, the signal to noise ration for the differential temperature measurements worsens for large $T_0$. We mention that our choice of discretization ensures that each approximate Dirac heating pattern ($a=0.01$) covers a minimum of four grid points. 

The conductivity is assumed to be known for nodes that are $0.5$cm, or less, away from the boundary. The Tikhonov regularization parameter is $\gamma = 5^{-4}$, and iterations are run until the 2-norm of the step is less than $0.1$. The initial guess is the solution to \eqref{eqn:syst_disc_real} with a constant conductivity. 

\begin{figure}[h]
	\centering 
	\begin{tabular}{ccc}
	\includegraphics[width=45mm]{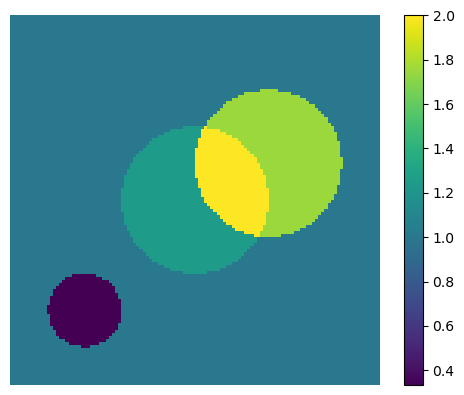} &
	\includegraphics[width=45mm]{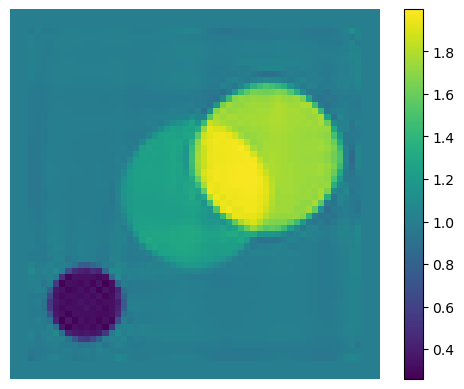} & \includegraphics[width=45mm]{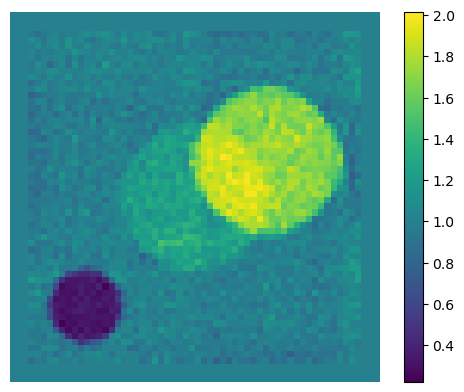} \\
	(a)   & (b) & (c) 
	\end{tabular}
	\caption{Reconstructions of purely real conductivity values (cm$^{-1}$ $k\Omega^{-1}$) on a $10$cm$\times$10cm square domain. The ground truth conductivity in (a) is evaluated on the fine grid. Both the reconstructions using the deterministic model (b) and a stochastic model (c) are evaluated on the coarse grid.}
	\label{fig:det_real}
\end{figure}

The numerical example in \cref{fig:det_bc} uses data from experimental boundary conditions in \eqref{eq:aux2}.The set $\Gamma$ defined in \eqref{eqn:gamma}, corresponds to the electrode functions. On $\Gamma$, we use electrode functions $g_1$ and $h_1$ \eqref{eqn:bc} for the boundary conditions. This set has gaps of $1$cm at the center of each side of the square with no flux conditions. The no flux conditions are enforced by using centered approximations to the nodes on $\partial \Omega - \Gamma$ (see, e.g., \cite[sec 2.12]{Leveque:2007:FDM}). The ground truth conductivity is given in \cref{fig:det_bc} (a). The reconstructions are evaluated on a $100 \times 100$ coarse grid, and a $200 \times 200$ fine grid is used to evaluate the measurements \eqref{eqn:det_meas}. A minimum of twelve fine grid points are in the effective area of each approximate Dirac heating pattern. 

The Gauss-Newton iteration is regularized with $\gamma = 3^{-3}$, and iterations are run until the 2-norm of the step size is less than $0.1$. The reconstructions in \cref{fig:det_bc} are close to the original conductivity, although there are some numerical artifacts due to the gaps between the electrodes. 

\begin{figure}[h!]
	\centering 
	\begin{tabular}{ccc}
		\includegraphics[width=45mm]{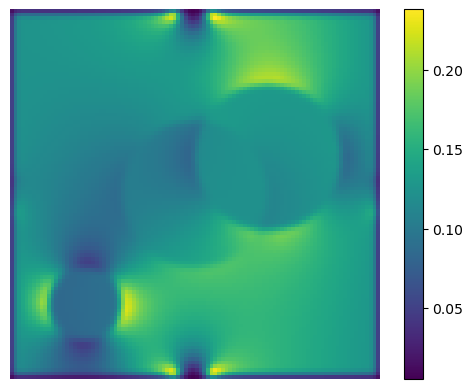} &
		\includegraphics[width=45mm]{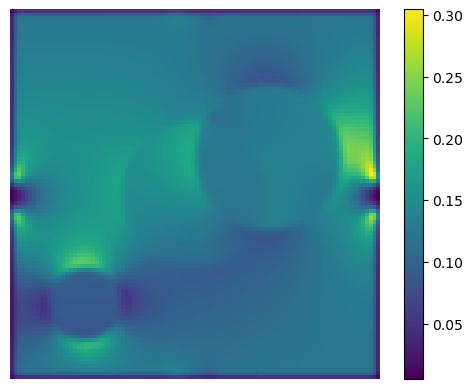} & 
		\includegraphics[width=45mm]{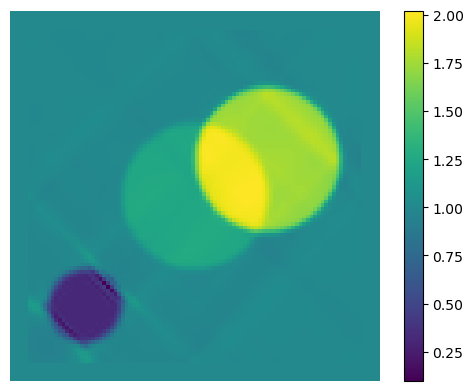} \\
		(a)  & (b)  & (c) 
	\end{tabular}
	\caption{Conductivity values (cm$^{-1}$ $k\Omega^{-1}$) for a numerical reconstruction (c) of the same conductivity as \cref{fig:det_real} using the experimental boundary conditions in \cref{eq:aux2}. The data $H_{ii} =\Re(\sigma(x))|\nabla u_i(x)|^2$ used for the reconstructions is shown in (a) for $g_1$ and in (b) for $h_1$, with units  cm$^{-3}$ $k\Omega^{-1}$.}
	\label{fig:det_bc}
\end{figure}

\subsection{Complex conductivity}\label{sec:complex} 
An example of a complex conductivity reconstruction using the deterministic model of measurements can be seen in \cref{fig:complex_rep}. Four experiments are used with the Dirichlet boundary conditions $\tilde{g}_{1}, \tilde{g}_{2}, \tilde{h}_{1}, \tilde{h}_{2}$ given in \eqref{eqn:bc_com}. The reconstructions are evaluated on a $100 \times 100$ coarse grid, and a $200 \times 200$ fine grid is used to evaluate the measurements \eqref{eqn:det_meas}. The Gauss-Newton iteration now uses $\gamma = 1^{-4}$, and iterations are run until the 2-norm of the step size is less than $0.1$. A constant complex conductivity and the corresponding solutions to \eqref{eqn:syst_disc_complex} are used for the initial guess. 

The numerical artifacts in the complex reproduction are consistent with the areas in \cref{fig:sym_complex} (c), where the maximum condition number \eqref{eqn:condition} of the symbol is largest. Intuitively, reconstructing the imaginary conductivity may be more challenging as it does not explicitly appear in the measurements. In the discrete system \eqref{eqn:syst_disc_complex} it appears only when coupled with a gradient of the real or imaginary auxiliary field. When the gradient of the real part is large it may overwhelm the contribution of the complex conductivity. This is the exact behavior we see in our numerical experiments. When the real conductivity is high, the gradient of the real field is large, and the reconstructed complex conductivity (\cref{fig:complex_rep} (d)) is lower than the true value. 

\begin{figure}[h]
	\centering
	\begin{tabular}{cc}
		\includegraphics[width=55mm]{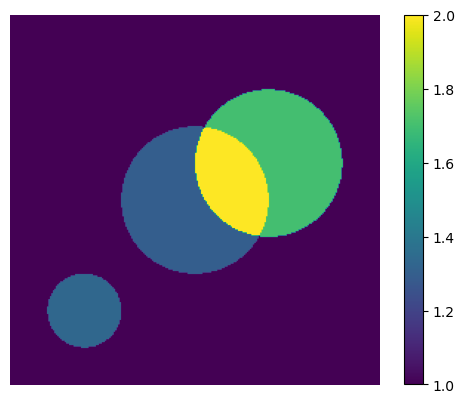} & 
		\includegraphics[width=55mm]{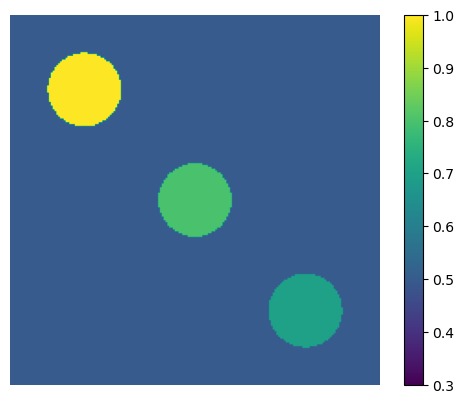} \\
		(a) True $\sigma'$ & (b) True $\epsilon \omega$ \\
		\includegraphics[width=55mm]{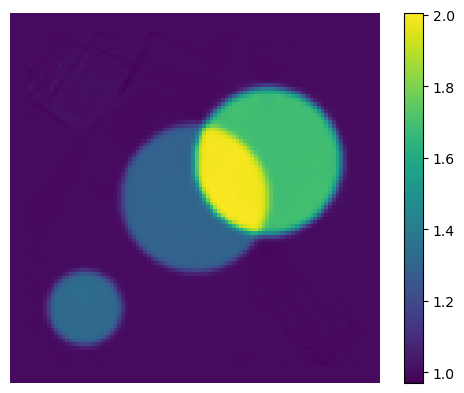} & 
		\includegraphics[width=55mm]{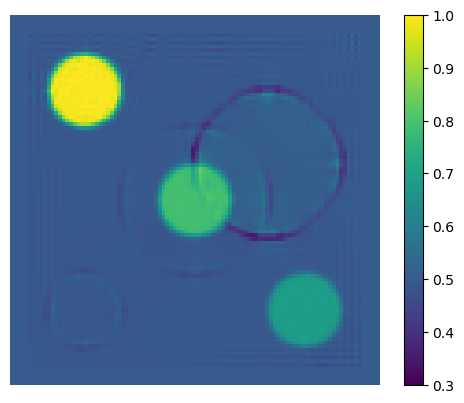} \\
		(c) Reconstructed $\sigma'$ & (d) Reconstructed $\epsilon \omega$
	\end{tabular}
	\caption{Real and imaginary conductivity values (cm$^{-1}$ $k\Omega^{-1}$) for a numerical reconstruction (c) \& (d) of a complex conductivity on a $10$cm$\times$10cm square domain. The ground truth (a) \& (b) is evaluated on the fine grid and the reconstruction is evaluated on the coarse grid. }\label{fig:complex_rep}
\end{figure}

%%%%%%%%%%%%%%%%%%%%%%%%%%%%%%%%%%%%%%%%%%%%%%%%%%%%%%%%%%%%%%%%%%%%%%%
\section{Summary and perspectives}
\label{sec:summary}
We propose a new hybrid inverse problem for recovering the conductivity of a body using thermal noise. The fluctuation dissipation theorem for electrodynamic media allows us to relate the variance of thermal noise currents taken with different temperature patterns to the real part of the conductivity of a body. By taking a sufficiently rich set of measurements we can estimate an internal functional that depends on this real conductivity and the solution to an associated auxiliary problem. We show this relation holds for both Dirichlet boundary conditions and mixed Neumann/Dirichlet boundary conditions, the latter of which is a more realistic description of an experimental setup where such measurements might be used. For purely real conductivities, these are power density measurements. This problem of recovering a real conductivity from power density measurements also appears in acousto-electric tomography. 

Before attempting numerical reconstructions we try and determine if the linearized problems are elliptic in the Douglis-Nirenberg sense. The linearized real problem has previously been shown to be elliptic, given the auxiliary fields are nowhere orthogonal or parallel \cite{Bal:2014:HIP}. For the real problem with mixed boundary conditions, we make no effort to show our boundary conditions satisfy this condition. Instead, we numerically evaluate the worst conditioning of the principal symbol at each point in space. The numerical results are consistent with the previous theoretical work in \cite{Bal:2014:HIP}; using more boundary conditions improves the numerical conditioning of the symbol. For the complex symbol, we give a sufficient condition on the auxiliary fields for the problem to be elliptic, without proving that the auxiliary fields can be generated. We perform a similar numerical evaluation of the conditioning of the symbol under a number of different experiments. This evidence indicates that the complex conductivity problem with two boundary conditions is not elliptic. This numerical approach to classification may find use in similar problems, especially in problems with complicated boundary conditions or principal symbols. The clear limitation of this method is that it does not inform the choice of boundary conditions. We note that the conditioning in these problems would not normally be seen as high for other applications. 

Finally, we present a simple discrete model for numerical reconstructions. Our numerical reconstructions are consistent with the linearization study. We also present results using simulated random thermal currents for the case of a purely real conductivity. These simulations ignore temporal correlations, making the problem more challenging. In this case, we can get an accurate, if noisy, reconstruction of the conductivity at a low temperature.   

This method of thermal noise imaging may find applications in e.g. Atomic Force Microscopy, laser weld monitoring. A challenge in using our approach is that the relative size of the measurements due to the background temperature and the heating pattern (see \cref{rem:measure_size}) results in currents that may hard to measure reliably in practice.

Our results are for a fixed frequency $\omega$, and removing this limitation may allow for more accurate reconstructions in the complex case. Considering multiple frequencies means that $\epsilon$ will be treated as a variable instead of $\epsilon \omega$. This is an important difference because it changes the structure of the symbol and requires a separate analysis. Additionally, the relative scale of the variables of interest, $\omega$ and $\epsilon$, changes which may introduce other challenges to the reconstructions. 

%%%%%%%%%%%%%%%%%%%%%%%%%%%%%%%%%%%%%%%%%%%%%%%%%%%%%%%%%%%%%%%%%%%%%%%
\section*{Acknowledgments}
The authors would like to thank Maxence Cassier for useful comments on a draft of this manuscript.

\bibliographystyle{siamplain}
\bibliography{thermal}

\end{document}